\documentclass[11pt,oneside]{article}
 
\usepackage{cite}

\usepackage[intlimits]{amsmath}
\usepackage{amsthm,amsfonts}

\usepackage{url}

\renewcommand{\Hat}{\widehat}

\newcommand{\eps}{\varepsilon}

\newcommand{\RR}{\mathbb{R}}

\newcommand{\CC}{\mathbb{C}}
\newcommand{\NN}{\mathbb{N}}

\newcommand{\cD}{\mathcal{D}}
\newcommand{\cH}{\mathcal{H}}
\newcommand{\cG}{\mathcal{G}}

\newcommand{\cB}{\mathcal{B}}
\newcommand{\cL}{\mathcal{L}}
\newcommand{\cK}{\mathcal{K}}

\newcommand{\cA}{\mathcal{A}}
\newcommand{\ELL}{\mathcal{L}}

\DeclareMathOperator{\Id}{Id}

\newtheorem{theorem}{Theorem}
\newtheorem*{theorem*}{Theorem}
\newtheorem{prop}[theorem]{Proposition}
\newtheorem{lemma}[theorem]{Lemma}
\newtheorem{corol}[theorem]{Corollary}

\theoremstyle{remark}

\newtheorem{remark}[theorem]{Remark}

\theoremstyle{definition}

\newtheorem{defin}[theorem]{Definition}

\newcommand{\slim}{\mathop{\text{s-lim}}}

\DeclareMathOperator{\supp}{supp}
\DeclareMathOperator{\spec}{spec}
\DeclareMathOperator{\dom}{dom}
\DeclareMathOperator{\ran}{ran}

\newcommand{\uhr}{\mathbin{\lceil}}

\begin{document}

\title{\bf An example of unitary equivalence between self-adjoint extensions and their parameters}

\author{{\sc Konstantin Pankrashkin}\\[\bigskipamount]
Laboratoire de math\'ematiques -- UMR 8628\\
Universit\'e Paris-Sud 11, B\^atiment 425\\
91400 Orsay, France\\
Homepage: \url{http://www.math.u-psud.fr/~pankrash/}\\
E-mail: \url{konstantin.pankrashkin@math.u-psud.fr}
}

\date{}
\maketitle

\begin{abstract}
\noindent The spectral problem for self-adjoint extensions is studied using the machinery of boundary triplets.
For a class of symmetric operators having Weyl functions of a special type we calculate explicitly
the spectral projections in the form of operator-valued integrals.
This allows one to give a constructive proof of the fact that, in certain intervals,
the resulting self-adjoint extensions are unitarily equivalent to the parameterizing boundary operator
acting in a smaller space, and one is able to provide an explicit form for the associated unitary transform.
Applications to differential operators on metric graphs and to direct sums are discussed.

\medskip


\medskip

\end{abstract}

\section{Introduction}

Let $S$ be a closed densely defined symmetric operator
in a separable Hilbert space $\cH$ with equal deficiency indices
$n_\pm(S)=n$; it is a well known fact that the self-adjoint extensions
of $S$ are then parameterized by the unitary operators or self-adjoint linear relations
in the $n$-dimensional Hilbert space \cite{GG}.
From the point of view of quantum mechanics, the choice of a particular self-adjoint extension
can be viewed as a quantization problem, as different extensions correspond to Hamiltonians
defining different quantum dynamics \cite{gitman}, and understanding the role of the extension
parameters
in the properties of the associated Hamiltonians is one of the central problems.
Questions of the control of the point spectrum in terms of parameters
were addressed already by Krein \cite{krein} for the case of finite deficiency indices,
and it was then extended in a series of papers to more sophisticated situations
including the study of singular and absolutely continuous spectra, see the discussion
in \cite{ABN,jfb}. The progress in the spectral analysis of self-adjoint extensions
was mainly possible due to the concept
of boundary triplet and associated Weyl functions \cite{GG,DM,DM2}, whose main
points we briefly recall now. A \emph{boundary triplet} \cite{GG} for $S$ consists of an auxiliary
Hilbert space $\cG$
and two linear maps $\Gamma,\Gamma':\dom S\to\cG$ satisfying the following two conditions:
\begin{itemize}
\item $\langle f,S^*g\rangle_\cH-\langle S^*f,g\rangle_\cH =\langle \Gamma f,\Gamma'g\rangle_\cG-\langle \Gamma'f,\Gamma g\rangle_\cG$
for all $f,g\in\dom S^*$,
\item the application $(\Gamma,\Gamma'):\dom S^*\ni f \mapsto (\Gamma f,\Gamma' f)\in \cG\oplus\cG$
is surjective.
\end{itemize}
The boundary triplets provide an efficient way of dealing with the self-adjoint extensions
of the operator $S$. We restrict ourselves by considering two distinguished self-adjoint extensions, namely,
\[
H^0:=S^*\uhr{\ker \Gamma},
\quad
H:=S^*\uhr{\ker\Gamma'};
\]
here and below $A\uhr L$ denotes the restriction of $A$ to $L$.
In our setting $H^0$ will be considered as a reference operator, and $H$ is viewed as its perturbation, and 
we will show below that various situations are reduced to the above case,
and the parameters may be included into the boundary triplet.

An essential role in the analysis of the self-adjoint
extensions is played by the so-called Weyl function $M(z)$ which is defined as follows \cite{DM}.
It is known that for $z\notin\spec H^0$ the operator
\[
\gamma(z):=\big(\Gamma\uhr{\ker(S^*-z)}\big)^{-1}
\]
is well defined, and it is a linear topological isomorphism between $\cG$ and $\ker(S^*-z)\subset\cH$.
The map
\[
\CC\setminus\spec H^0\ni z\mapsto \gamma(z)\in \cL(\cG,\cH),
\]
which is usually called the $\gamma$-field, is holomorph.
The holomorph operator function 
\[
\CC\setminus\spec H^0\ni z\mapsto M(z):=\Gamma'\gamma(z)\in \cL(\cG)
\]
is referred to as the Weyl function associated with the boundary triplet.
The following result is well known \cite{DM}:
\begin{theorem}
      \label{th1}
For any $z\notin \spec H^0$ there holds
\begin{equation}
      \label{eq-ker}
\ker(H-z)=\gamma(z)\ker M(z).
\end{equation}
Moreover, for $z\notin \spec H^0\cup\spec H$ the Krein resolvent formula holds,
\begin{equation}
        \label{eq-krein}
(H-z)^{-1}=(H^0-z)^{-1}-\gamma(z)M(z)^{-1}\gamma(\Bar z)^*.
\end{equation}
\end{theorem}
A direct consequence of the resolvent formula \eqref{eq-krein} is the relation
\begin{equation}
        \label{eq-hspec}
\spec H\setminus\spec H^0=\Big\{
z\notin\spec H^0:\, 0\in\spec M(z)
\Big\}.
\end{equation}
Numerous papers were dedicated to the study of possible refinements of \eqref{eq-hspec}
in order to recover the spectral nature of $H$ in terms of the Weyl function, see e.g. the discussion in~\cite{ABMN,jfb,BMN02,DM,DM2,MN1}.
Such an analysis becomes deeper if some additional properties
of $H^0$ and $M$ are available. In the present paper we assume that
the following two conditions hold:
\begin{itemize}
\item the operator $H^0$ has a spectral gap, i.e. there exists
an open interval $J:=(a_0,b_0)\subset \RR\setminus\spec H^0$, and
\item the Weyl function $M$ can be represented as
\begin{equation}
          \label{eq-mspec}
M(z)=\dfrac{m(z)-T}{n(z)},
\end{equation}
where $T$ is a bounded self-adjoint operator in $\cG$ 
and $m$, $n$ are scalar
functions which are holomorph outside $\spec H^0$, with $n$ non-vanishing in $J$.
\end{itemize}
Such situations arise in several important applications such as the analysis of differential operators
on metric graphs, this will  be reviewed in Section \ref{sec-appl}.
In this case, Eq. \eqref{eq-hspec} gives a rather simple characterization of the spectrum of $H$
in $J$:
\[
\spec H\cap J=\big\{ \lambda\in J:\, m(\lambda)\in\spec T\big\}.
\]
This relation was refined in \cite{BGP08} in order to describe
all spectral types of $H$:
\begin{equation}
       \label{eq-spss}
\spec_\star H\cap J=\big\{ \lambda\in J:\, m(\lambda)\in\spec_\star T\big\}, \quad
\star\in\{\text{p},\text{pp},\text{disc},\text{ess},\text{ac},\text{sc}\}.
\end{equation}
A further improvement was obtained first in \cite{ABMN} for the particular case $n=\text{const}$
and then extended in \cite{KP11} to general denominators $n$, and we need to introduce
an additional notation to formulate it. Let $\cB(\RR)$ denote the set of the borelian subsets of $\RR$.
For a self-adjoint operator $A$ acting in a Hilbert space $\cH$ we denote
by $E_A:\cB(\RR)\to \cL(\cH)$ the operator-valued spectral measure associated with $A$,
i.e. for any $\Omega\in \cB(\RR)$ we have $E_A(\Omega):=1_\Omega(A)$, where
$1_\Omega:\RR\to\RR$ is the indicator function of $\Omega$. For the same $\Omega$,
we denote by $A_\Omega$ the operator $A E_A(\Omega)$ viewed as a self-adjoint operator
in the Hilbert space $\ran E_A(\Omega)$ equiped with the induced scalar product.
Using this notation, the main result of \cite{ABMN,KP11}
can be formulated as follows: If the Weyl function
admits the representation \eqref{eq-mspec}, then $m:J\cap\spec H\to m(J\cap\spec H)$
is a bijection, and the operator $H_J$ is unitarily equivalent to $m^{-1}(T_{m(J)})$.
This indeed implies the spectral relations \eqref{eq-spss}.

One has to emphasize that the approach suggested in \cite{ABMN} and then used in \cite{KP11}
was rather implicit. In particular, the unitary equivalence of the two operators
was shown using the so-called generalized Naimark theorem
on minimal orthogonal dilations of operator-valued measures \cite{MM}, and
no information on the unitary transform was obtained.
Finding this transform
in an explicit form is the subject of the present work.
The main results of the present paper can be summarized as follows (see Theorem \ref{thm10} and Corollary \ref{corol11}):
\begin{itemize}
\item for any borelian subset $\Omega\subset J$ one has the equality
$E_H(\Omega)= U E_T\big(m(\Omega)\big)U^*$,
where $U$ is the operator given by
\[
U = \int_J \sqrt{\dfrac{n(\lambda)}{m'(\lambda)}}\, \gamma(\lambda) dE_T\big(m(\lambda)\big)
\]
understood as an improper operator-valued Riemann-Stieltjes integral (the precise meaning will be discussed in detail in Sections \ref{sec-proj}
and \ref{sec-unt}).
\item Moreover, $U$ viewed as a map from $\ran E_T\big(m(J)\big)$ to $\ran E_H(J)$
is unitary, and $m(H_J)= U T_{m(J)} U^*$.
\end{itemize}
As we will see below, in various situations the operator $T$ plays the role of a parameter
of self-adjoint extensions, so the above results provide a direct translation of the spectral properties
of the parameters to the spectral properties of the associated self-adjoint extensions.

We believe that the knowledge of a certain explicit form for the unitary transform $U$ relating
the operators $H$ and $T$ provides a considerable progress compared to the previously known results.
In particular, it allows one to reduce the functional calculus for $H$ to that for $T$ and,
for example, to make a link between various evolution problems associated with the two operators. 
It should be emphasized that the approach used in the present paper differs from the one employed
in the previous works \cite{ABMN,KP11},
we calculate the spectral projections for $H$ in a rather direct way using the limit values of the resolvent,
which allows one
to write an explicit formula for the unitary transform in question.
While the situation we look at is indeed very special, we show below that it admits
some useful applications in mathematical physics, see Section~\ref{sec-appl}.
Moreover, we are not aware of any previous work giving any explicit expression
for the spectral projections of suitably large classes of self-adjoint extensions in a closed form,
and we believe that the approach presented here may admit some generalizations to more sophisticated
Weyl functions and might shed a new light on the spectral analysis and the functional calculus of self-adjoint extensions.

\section{Basic formula for spectral projections}\label{sec-proj}

The aim of the present section is to express the spectral projections for $H$ in terms
of the spectral measure for  $T$. The main formula is given in Lemma \ref{prop3}, and it will be used
for the subsequent spectral analysis. Throughout the rest of the paper we use the notation
\[
S_T:=[\inf\spec T,\sup\spec T]\subset\RR.
\]

The following simple properties of $m$ and $n$ were proved in \cite[Section 2.2]{KP11}:
\begin{lemma} \label{lem1}
Denote $K:=m^{-1}(S_T)\cap J$, then
\begin{itemize}
\item $n(x)\cdot m'(x)> 0$ for all $x\in K$,
\item $K$ is a connected set.
\end{itemize}
\end{lemma}

The maps $\gamma(z)$ and $M(z)$ satisfy a number of identities, see e.g. \cite{BGP08,DM}.
In particular, 
\begin{gather}
      \label{eq-mconj}
M(z)^*=M(\overline z), \quad z\notin\spec H^0,\\
       \label{eq-mgam}
M(z_1)-M(\overline{z_2})=(z_1-\overline{z_2})\gamma(z_2)^*\gamma(z_1), \quad
z_1,z_2\notin\spec H^0,\\
\Im M(z)>0 \text{ for } \Im z>0. \nonumber
\end{gather}
The last property means that $M$ is a strict operator-valued Nevanlinna-Herglotz function.
Recall that for any operator-valued Nevanlinna-Herglotz function $F$
in $\cG$ and any $f\in\cG$, the map 
\[
\CC_+:=\{z\in\CC:\, \Im z>0\}\ni z\mapsto F_f(z):= \big\langle f,F(z)f\big\rangle_\cG
\]
is a scalar Nevanlinna-Herglotz
function (called also $R$-function), hence for almost all
$x\in\RR$ there exists a finite  limit
$\lim_{\varepsilon\to 0+} \big\langle f,F(x+i\varepsilon)f\big\rangle_\cG$.
Using the polar identity we see that,
for any $f,g\in\cG$, the limit $\lim_{\varepsilon\to 0+} \big\langle f,F(x+i\varepsilon)g\big\rangle_\cG$
is finite for almost all $x\in\RR$.

The following proposition will not be used in the rest of the text, but
it provides us  with a certain intuition in the study of the spectral projections for $H$.
\begin{prop}\label{prop1}
For any $\lambda\in J\mathop{\cap}\spec_\mathrm{p} H$ there holds
\[
E_H\big(\{\lambda\}\big)=\dfrac{n(\lambda)}{m'(\lambda)}\,\gamma(\lambda)E_T\Big(\big\{m(\lambda)\big\}\Big)\gamma(\lambda)^*.
\]
\end{prop}

\begin{proof}
Let $(e_j)$ be an orthonormal basis
in $\ker \big( T-m(\lambda)\big)$. It follows from \eqref{eq-ker}
that the family
$\Big( \sqrt{\dfrac{n(\lambda)}{m'(\lambda)}}\, \gamma(\lambda)e_j\Big)$
is a basis in $\ker(H-\lambda)$, and we just need to show that this basis
is orthonormal. To see this, take any $\xi_1,\xi_2\in \ker \big( T-m(\lambda)\big)$,
then, using \eqref{eq-mgam},
\begin{multline*}
\bigg\langle \sqrt{\dfrac{n(\lambda)}{m'(\lambda)}}\,\gamma(\lambda)\xi_1, \sqrt{\dfrac{n(\lambda)}{m'(\lambda)}}\,\gamma(\lambda)\xi_2\bigg\rangle_\cH
=
\dfrac{n(\lambda)}{m'(\lambda)}\,\big\langle\xi_1, \gamma(\lambda)^*\gamma(\lambda)\xi_2\big\rangle_\cG\\
=
\dfrac{n(\lambda)}{m'(\lambda)}\,\big\langle\xi_1, M'(\lambda)\xi_2\big\rangle_\cG
=
\dfrac{n(\lambda)}{m'(\lambda)}\,
\Big\langle
\xi_1, \Big( \dfrac{m'(\lambda)}{n(\lambda)} -\dfrac{n'(\lambda)}{n(\lambda)^2}\cdot
\big(m(\lambda)-T\big)\Big)\xi_2
\Big\rangle_\cG\\
=
\dfrac{n(\lambda)}{m'(\lambda)}\,
\Big\langle
\xi_1,  \dfrac{m'(\lambda)}{n(\lambda)}\, \xi_2\Big\rangle_\cG
=\langle\xi_1,\xi_2\rangle_\cG.
\end{multline*}
which concludes the proof.
\end{proof}
The aim of the present section is to find a certain  analog of
Proposition \ref{prop1} for operators with continuous spectra.
Take an arbitrary segment $[a,b]\subset J$ whose endpoints satisfy
\begin{equation}
       \label{eq-cond}
a,b\notin\spec_\mathrm{p} H  \text{ or, equivalently, }  m(a), m(b) \notin\spec_\mathrm{p} T,
\end{equation} 
and consider the operator
\begin{equation}
  \label{eq-pint}
\Phi\big([a,b]\big):= \int_a^b \sqrt{\dfrac{n(\lambda)}{m'(\lambda)}}\, \gamma(\lambda)\, dE_T\big(m(\lambda)\big) \in \cL(\cG,\cH)
\end{equation}
and its adjoint
\[
\Phi^*\equiv\Phi\big([a,b]\big)^*:= \int_a^b \sqrt{\dfrac{n(\lambda)}{m'(\lambda)}} \, dE_T\big(m(\lambda)\big)\, \gamma(\lambda)^*
\in \cL(\cH,\cG).
\]
These two operators are defined as suitable limits of the associated Riemann-Stieltjes integral sums.
More precisely, consider a partition $\Delta$ of $[a,b]$, $a=\lambda_0<\lambda_1<\dots<\lambda_n=b$,
denote $\Delta_j:=[\lambda_{j-1},\lambda_j)$, pick some $\xi_j\in \overline{ \Delta_j}$
and consider the Riemann-Stieltjes integral sum
\begin{equation}
      \label{eq-phisum}
\Phi_\Delta=\sum_{j=1}^n \sqrt{\dfrac{n(\xi_j)}{m'(\xi_j)}} \,\gamma(\xi_j) E_T\big(m(\Delta_j)\big).
\end{equation}
The strong limit of $\Phi_\Delta$ as $\max_j |\Delta_j|\to 0$, if it exists
and is independent of the choice of the partition points $\lambda_j$, $\xi_j$, 
is then denoted by the above integral expression \eqref{eq-pint}. The general theory
of such integrals is rather involved, see e.g. \cite{AM,BS}, but in our case
the maps $\gamma$ are holomorph, and the study of the above integral reduces to
the study of usual (scalar) Riemann-Stieltjes integral with respect to spectral measures.
The existence in our case can be proved in an easier way using the constructions of \cite[Section 7]{adam}.
Note that the points $m(\lambda)$ with $m'(\lambda)=0$ do not belong to the support of $E_T$ due to Lemma \ref{lem1}
and that the map $\lambda\mapsto \sqrt{\dfrac{n(\lambda)}{m'(\lambda)}}\,\gamma(\lambda)$
is lipschitzian on $[a,b]\cap m^{-1}(S_T)$. This is sufficient to show
that, for $\max_j |\Delta_j|\to 0$, the above integral sums $\Phi_\Delta$ and their adjoints
$\Phi_\Delta^*$ converge in the operator norm to their limit values
$\Phi\equiv\Phi\big([a,b]\big)$ and $\Phi\big([a,b]\big)^*$ respectively, see \cite[Remark 7.3]{adam}. 
The following lemma is the main result of the section and provides the main technical ingredient
of the subsequent discussion.
\begin{lemma}\label{prop3}
Let a segment $[a,b]\subset J$ satisfy \eqref{eq-cond}, then
$E_H\big([a,b]\big)=\Phi\big([a,b]\big)\Phi\big([a,b]\big)^*$.
\end{lemma}

\begin{proof}
Throughout the proof we denote for brevity
\[
P:=E_H\big([a,b]\big), \quad
\Phi:=\Phi\big([a,b]\big)
\quad\text{and}\quad \Pi:=\Phi \Phi^*.
\]
So we need to prove the equality $P=\Pi$.
By the Stone formula for the spectral projections, cf. \cite[Theorem 42]{Jak},
we have 
\[
P=\slim_{\eps\to +0}\dfrac{1}{\pi}
\int_a^b \Im (H-x-i\eps)^{-1}dx,
\]

Let us assume first that
\begin{equation}
        \label{eq-mp0}
m'(x)\ne 0 \text { for all } x \in[a,b],
\end{equation}

As the map $z\mapsto (H^0-z)^{-1}$ is holomorph in a complex neighborhood of $[a,b]$,
one has
\[
\slim_{\eps\to +0}\dfrac{1}{\pi}
\int_a^b \Im (H^0-x-i\eps)^{-1}dx=0.
\]
Using  the resolvent formula \eqref{eq-krein} we arrive at
\begin{equation}
          \label{eq-pgg}
P=-\slim_{\varepsilon\to +0}\dfrac{1}{\pi}
\int_a^b \Im\Big[ \gamma(x+i\varepsilon)M(x+i\varepsilon)^{-1}\gamma(x-i\varepsilon)^*\Big]dx.
\end{equation}
Let $\cD$ be the algebraic linear hull of the set $\big\{\gamma(z)\xi: z\in\CC\setminus\RR, \, \xi\in\cG\big\}$
and let $f\in \cD^\perp$. First, for all $z\notin\RR$ there holds $\gamma(z)^*f=0$, and Eq. \eqref{eq-pgg}
gives $Pf=0$.  On the other hand, the map $z\mapsto \gamma(z)^*f$ extends by continuity to $z\in [a,b]$,
which shows that $\gamma(\lambda)^* f=0$ for all $\lambda\in [a,b]$, which, in its turn, gives the equalities
$\Phi^*f=0$ and $\Pi f=0$.
Therefore, the operators $P$ and $\Pi$ coincide on $\cD^\perp$.

Now let us show that $P$ coincides with $\Pi$ on the closure of $\cD$.
As both $P$ and $\Pi$ are bounded operators,
it is now sufficient to prove that $\langle f,P g\rangle_\cH=\langle f, \Pi g\rangle_\cH$
for the vectors $f$ and $g$ having the form $f=\gamma(z_1)\xi_1$ and $g=\gamma(z_2){\xi_2}$
with arbitrary $z_1,z_2\in\CC\setminus\RR$ and $\xi_1,{\xi_2}\in\cG$. 
For such $f$ and $g$ we have
\begin{multline*}
\big\langle f,Pg\big\rangle_\cH
= -\lim_{\varepsilon\to +0}\dfrac{1}{\pi}
\int_a^b \Big \langle\gamma(z_1){\xi_1},\Im\Big[ \gamma(x+i\varepsilon)M(x+i\varepsilon)^{-1}\gamma(x-i\varepsilon)^*\Big]\gamma(z_2){\xi_2}\Big\rangle_\cH
dx\\
=
-\dfrac{1}{\pi}
\lim_{\varepsilon\to +0}
\int_a^b \Big \langle \xi_1,\gamma(z_1)^*\Im\Big[ \gamma(x+i\varepsilon)M(x+i\varepsilon)^{-1}\gamma(x-i\varepsilon)^*\Big]\gamma(z_2){\xi_2}\Big\rangle_\cG
dx.
\end{multline*}
Using \eqref{eq-mconj} and \eqref{eq-mgam} we have
\begin{multline*}
2i\gamma(z_1)^*\Im\Big[ \gamma(x+i\varepsilon)M(x+i\varepsilon)^{-1}\gamma(x-i\varepsilon)^*\Big]\gamma(z_2)\\
\begin{aligned}
=&\gamma(z_1)^*\gamma(x+i\varepsilon)M(x+i\varepsilon)^{-1}\gamma(x-i\varepsilon)^*\gamma(z_2)\\
&-\gamma(z_1)^*\gamma(x-i\varepsilon)M(x-i\varepsilon)^{-1}\gamma(x+i\varepsilon)^*\gamma(z_2)\\
=&\dfrac{M(x+i\varepsilon)-M(\overline z_1)}{x+i\varepsilon-\overline z_1}\,M(x+i\varepsilon)^{-1}\dfrac{M(z_2)-M(x+i\varepsilon)}{z_2-x-i\varepsilon}\\
&-\dfrac{M(x-i\varepsilon)-M(\overline z_1)}{x-i\varepsilon-\overline z_1}\,M(x-i\varepsilon)^{-1}\dfrac{M(z_2)-M(x-i\varepsilon)}{z_2-x+i\varepsilon}\\
=&\dfrac{M(z_2)+M(\overline z_1) - M(x+i\varepsilon) -M(\overline z_1)M(z_2)M(x+i\varepsilon)^{-1}}{(x+i\varepsilon-\overline z_1)(z_2-x-i\varepsilon)}\\
&-\dfrac{M(z_2)+M(\overline z_1) - M(x-i\varepsilon) -M(\overline z_1)M(z_2)M(x-i\varepsilon)^{-1}}{(x-i\varepsilon-\overline z_1)(z_2-x+i\varepsilon)}.
\end{aligned}
\end{multline*}
As the map $z\mapsto M(z)$ is holomorph in a certain complex neighborhood of $[a,b]$
and $z_1$ and $z_2$ are fixed non-real numbers, we have the obvious relations
\begin{align*}
\lim_{\varepsilon\to +0} \int_a^b
\dfrac{M(z_2)+M(\overline z_1) - M(x+i\varepsilon)}{(x+i\varepsilon-\overline z_1)(z_2-x-i\varepsilon)}\,dx
&=
\lim_{\varepsilon\to +0} \int_a^b
\dfrac{M(z_2)+M(\overline z_1) - M(x-i\varepsilon)}{(x-i\varepsilon-\overline z_1)(z_2-x+i\varepsilon)}\,dx\\
&=\int_a^b 
\dfrac{M(z_2)+M(\overline z_1) - M(x)}{(x-\overline z_1)(z_2-x)}\,dx.
\end{align*}
Hence,
\begin{multline*}
\big\langle f,P g\big\rangle_\cH
=\dfrac{1}{2\pi i}\lim_{\varepsilon\to +0}
\int_a^b\Big\langle \xi_1, \Big[
\dfrac{M(\overline z_1)M(z_2)M(x-i\varepsilon)^{-1}}{(\overline z_1 -x+i\varepsilon)(z_2-x+i\varepsilon)}\\
-
\dfrac{M(\overline z_1)M(z_2)M(x+i\varepsilon)^{-1}}{(\overline z_1 -x-i\varepsilon)(z_2-x-i\varepsilon)}\Big]{\xi_2}\Big\rangle_\cG dx.
\end{multline*}
Furthermore, we have the representation
\[
\dfrac{1}{(\overline z_1 -x-i\varepsilon)(z_2-x-i\varepsilon)}=\dfrac{1+\varepsilon h(x,\varepsilon)}{(\overline z_1 -x)(z_2-x)},
\]
where $h$ is a continuous function with $h(x,0)=1$ and
such that for some $\varepsilon_0>0$ one has the bound
$h_0:=\sup_{x\in[a,b],\, |\varepsilon|<\varepsilon_0} \big|h(x,\varepsilon)\big|<\infty$.
We have then
\begin{equation}
\begin{aligned}
\big\langle f,P g \big\rangle_\cG
=&\lim_{\varepsilon\to +0}
\int_a^b\Big\langle\xi_1, 
\dfrac{M(\overline z_1)M(z_2)}{(\overline z_1 -x)(z_2-x)}\cdot \dfrac{M(x-i\varepsilon)^{-1}-M(x+i\varepsilon)^{-1}}{2\pi i}{\xi_2}\Big\rangle_\cG \,dx\\
&-\lim_{\varepsilon\to +0}
\int_a^b  \dfrac{\varepsilon h(x,-\varepsilon)\Big\langle M(z_1){\xi_1},M(x-i\varepsilon)^{-1}M(z_2){\xi_2}\Big\rangle_\cG}{2\pi i(\overline z_1 -x)(z_2-x)} \, dx\\
&-\lim_{\varepsilon\to +0}
\int_a^b  \dfrac{\varepsilon h(x,\varepsilon)\Big\langle M(z_1){\xi_1},M(x+i\varepsilon)^{-1}M(z_2){\xi_2}\Big\rangle_\cG}{2\pi i(\overline z_1 -x)(z_2-x)}  \,dx.
\end{aligned}
              \label{eq-3t}
\end{equation}
By elementary considerations, see \cite[Lemma 3.14]{BGP08}, for small but non-zero $\eps$
we have the representation
\[
M(x+i\varepsilon)^{-1}=n(x+i\varepsilon)L(x,\varepsilon) \big(m(x)+i\varepsilon m'(x)-T\big)^{-1},
\]
where $L(x,\varepsilon)\in \cL(\cG)$ such that
$\lim_{\varepsilon\to 0}\sup_{x\in[a,b]} \big\|L(x,\varepsilon)-1\big\|_{\cL(\cG)}=0$.
By the assumption \eqref{eq-mp0} made at the beginning, we have the bound
\[
\kappa:=\inf_{x\in[a,b]} \big|m'(x)\big|>0.
\]
and for sufficiently small $\varepsilon\ne 0$ and all $x\in[a,b]$ one has the uniform estimate, with some fixed $C>0$,
\[
\Big|
\dfrac{\varepsilon h(x,\pm\varepsilon)}{(\overline z_1 -x)(z_2-x)} \Big\langle M(z_1){\xi_1},M(x\pm  i\varepsilon)^{-1}M(z_2){\xi_2}\Big\rangle_\cG
\Big|\le C.
\]
On the other hand, as noted at the beginning of the section, for almost all $x\in\RR$
there exist finite limits
\[
\lim_{\varepsilon\to 0+}\Big\langle M(z_1){\xi_1},M(x\pm i\varepsilon)^{-1}M(z_2){\xi_2}\Big\rangle_\cG=:C_\pm(x).
\]
Therefore, by using the dominated convergence we arrive at
\begin{multline*}
\lim_{\varepsilon\to +0}
\int_a^b  \dfrac{\varepsilon h(x,\pm\varepsilon)}{(\overline z_1 -x)(z_2-x)} \Big\langle M(z_1){\xi_1},M(x\pm i\varepsilon)^{-1}M(z_2){\xi_2}\Big\rangle_\cG dx
\\
\begin{aligned}
=&
\int_a^b  \lim_{\varepsilon\to +0}
\dfrac{\varepsilon h(x,\pm\varepsilon)}{(\overline z_1 -x)(z_2-x)} \Big\langle M(z_1){\xi_1},M(x\pm i\varepsilon)^{-1}M(z_2){\xi_2}\Big\rangle_\cG dx\\
=&\int_a^b  \lim_{\varepsilon\to +0}
\dfrac{ \varepsilon h(x,\pm\varepsilon)}{(\overline z_1 -x)(z_2-x)}
\lim_{\varepsilon\to +0}  \Big\langle M(z_1){\xi_1},M(x\pm i\varepsilon)^{-1}M(z_2){\xi_2}\Big\rangle_\cG dx\\
=&\int_a^b  \lim_{\varepsilon\to +0}
\dfrac{ \varepsilon h(x,\pm\varepsilon)}{(\overline z_1 -x)(z_2-x)} \,
C_\pm(x) dx=0.
\end{aligned}
\end{multline*}
Substituting these equalities into \eqref{eq-3t} we arrive at
\begin{equation}
    \label{eq-loc5}
\big\langle f,P g\big\rangle_\cH
=\lim_{\varepsilon\to +0}
\int_a^b\Big\langle \xi_1, 
\dfrac{M(\overline z_1)M(z_2)}{(\overline z_1 -x)(z_2-x)}\cdot \dfrac{M(x-i\varepsilon)^{-1}-M(x+i\varepsilon)^{-1}}{2\pi i}\,{\xi_2}\Big\rangle_\cG dx
\end{equation}
Denote by $\mu$ the spectral measure for the operator $T$ associated with the vectors
$M(z_1){\xi_1}$ and $M(z_2){\xi_2}$. With the help of $\mu$ one can rewrite the equality \eqref{eq-loc5} as follows:
\begin{multline}
\int_a^b\Big\langle \xi_1, 
\dfrac{M(\overline z_1)M(z_2)}{(\overline z_1 -x)(z_2-x)}\cdot \dfrac{M(x-i\varepsilon)^{-1}-M(x+i\varepsilon)^{-1}}{2\pi i}\,{\xi_2}\Big\rangle_\cG dx\\
\begin{aligned}
=&\int_a^b \dfrac{1}{(\overline z_1 -x)(z_2-x)}\Big\langle
M(z_1){\xi_1}, 
\dfrac{M(x-i\varepsilon)^{-1}-M(x+i\varepsilon)^{-1}}{2\pi i}
M(z_2){\xi_2}\Big\rangle_\cG\,dx\\
=&\int_a^b \int_{S_T}
\dfrac{1}{2\pi i(\overline z_1 -x)(z_2-x)} \Big(
\dfrac{n(x+i\varepsilon)}{\lambda-m(x+i\varepsilon)}-\dfrac{n(x+i\varepsilon)}{\lambda-m(x+i\varepsilon)}
\Big) d\mu(\lambda) dx\\
=&\int_{S_T} k(\lambda,\varepsilon) d\mu(\lambda),
\end{aligned}
   \label{eq-intk}
\end{multline}
where we denoted
\[
k(\lambda,\varepsilon) =\int_a^b\dfrac{1}{2\pi i(\overline z_1 -x)(z_2-x)}\cdot \Big(
\dfrac{n(x+i\varepsilon)}{\lambda-m(x+i\varepsilon)}-\dfrac{n(x-i\varepsilon)}{\lambda-m(x-i\varepsilon)}
\Big) dx.
\]
Our aim now is to pass to the limit $\varepsilon\to +0$ in the integral on the right-hand side of \eqref{eq-intk}.
By \cite[Lemma 10]{KP11} one has
\begin{equation*}
\dfrac{1}{\lambda-m(x+i\varepsilon)}=\dfrac{1}{\lambda -m(x) -i\varepsilon m'(x)}
\cdot\Big(
1 + \varepsilon\, g(x,\lambda,\varepsilon)
\Big),
\end{equation*}
with
\[
\sup_{\substack{x\in [a,b],\, \lambda\in \RR\\ 0<|\varepsilon|<\varepsilon_0}} \big|g(x,\lambda,\varepsilon)\big|<+\infty.
\]
As $n$ is a holomorph function, one has the representation $n(x+i\varepsilon)=n(x)+\varepsilon p(x,\varepsilon)$
with
\[
\sup_{\substack{x\in [a,b]\\ |\varepsilon|<\varepsilon_0}} \big|p(x,\varepsilon)\big|<+\infty.
\]
Hence one can represent the above function $k$ as
\begin{equation}
         \label{eq-kkk}
\begin{aligned}
k(\lambda,\varepsilon)&=\int_{a}^{b}
\dfrac{n(x)}{2\pi i(\overline z_1 -x)(z_2-x)}\cdot\Big(
\dfrac{1}{\lambda-m(x)-i\varepsilon m'(x)}
-
\dfrac{1}{\lambda-m(x)+i\varepsilon m'(x)}
\Big)dx
\\
&\quad+
\int_{a}^{b} 
\dfrac{1}{2\pi i(\overline z_1 -x)(z_2-x)}
\dfrac{\varepsilon r(x,\lambda, \varepsilon)}{\lambda-m(x)-i\varepsilon m'(x)}\,dx\\
&\quad +
\int_{a}^{b} 
\dfrac{1}{2\pi i(\overline z_1 -x)(z_2-x)}
\dfrac{\varepsilon r(x,\lambda, -\varepsilon)}{\lambda-m(x)+i\varepsilon m'(x)}\,dx
\end{aligned}
\end{equation}
with $r(x,\lambda,\varepsilon):= p(x,\varepsilon)\big(1+\varepsilon g(x,\lambda,\varepsilon)\big) + n(x) g(x,\lambda,\varepsilon)$,
and one has the obvious bound
\[
\sup_{\substack{x\in [a,b],\, \lambda\in \RR\\ 0<|\varepsilon|<\varepsilon_0}} \big|r(x,\lambda,\varepsilon)\big|=:R<+\infty.
\]
Denote the three summands on the right-hand side of \eqref{eq-kkk} by $I_1(\lambda,\varepsilon)$,
$I_2(\lambda,\varepsilon)$ and $I_3(\lambda,\varepsilon)$, respectively.
For all $\lambda\in\RR$ and $0<|\varepsilon|<\varepsilon_0$ we can estimate
\[
\Big|
\dfrac{\varepsilon r(x,\lambda, \varepsilon)}{\lambda-m(x)+i\varepsilon m'(x)}
\Big|\le \dfrac{R}{\kappa},
\]
implying the bound
\[
\big|
I_{2/3}(\lambda,\varepsilon)
\big|\le  \dfrac{R|b-a|}{2\pi \kappa\cdot |\Im z_1 \Im z_2|}
\text{ for all $\lambda\in\RR$ and $0<|\varepsilon|<\varepsilon_0$.}
\]
To study $I_1$, we rewrite the above expression in the form
\[
I_1(\lambda,\varepsilon)=
\dfrac{1}{\pi}\,\int_{a}^{b}
\dfrac{1}{(\overline z_1 -x)(z_2-x)}\cdot
\dfrac{\varepsilon m'(x) n(x)}{\big(\lambda-m(x)\big)^2 + \big(\varepsilon m'(x)\big)^2}\, dx.
\]
Denoting $N:=\sup_{x\in [a,b]} \big|n(x)\big|$ one obtains
\begin{multline*}
|I_1|\le \dfrac{N}{\pi|\Im z_1 \Im z_2|} \int_{a}^{b} \dfrac{\big|m'(x)\big|}{\big(\lambda-m(x)\big)^2 + \varepsilon^2 \kappa^2}\,dx
=\dfrac{N}{\pi|\Im z_1 \Im z_2|} \,\left|\,
\int_{m(a)}^{m(b)} \dfrac{\varepsilon}{(\lambda -y)^2+\varepsilon^2 \kappa^2}\,dy \right|\\
\le
\dfrac{N}{\pi|\Im z_1 \Im z_2|} \int_{-\infty}^{+\infty} \dfrac{\varepsilon}{y^2+\varepsilon^2 \kappa^2}\,dy
 =\dfrac{N}{\kappa|\Im z_1 \Im z_2|}.
\end{multline*}
Therefore, we have shown the estimate $\sup_{\lambda\in S_T,\, \varepsilon\in(0,\varepsilon_0)}\big| k(\lambda,\varepsilon)\big|<\infty$, and
the dominated convergence gives
\begin{equation} 
         \label{eq-proj}
\langle f, Pg\rangle_\cG=\lim_{\varepsilon\to +0} \int_{S_T} k(\lambda,\varepsilon) d\mu(\lambda)=
\int_{S_T} \lim_{\varepsilon\to +0} k(\lambda,\varepsilon) d\mu(\lambda).
\end{equation}
We are now going to calculate the limit $\lim_{\varepsilon\to +0} k(\lambda,\varepsilon)$.
Consider first the limits of the terms $I_2$ and $I_3$. As noted above, the subintegral
functions are uniformly bounded for small $\varepsilon$, and by applying the dominated convergence
we obtain
\[
\lim_{\varepsilon\to 0+} I_2(\lambda,\varepsilon)=
\int_{a}^{b} \dfrac{1}{(\overline z_1 -x)(z_2-x)}\,\lim_{\varepsilon\to 0+} \dfrac{\varepsilon r(x,\lambda, \varepsilon)}{\lambda-m(x)+i\varepsilon m'(x)}\, dx.
\]
For all but at most one $x\in[a,b]$ we have $\lambda\ne m(x)$ and
\[
\lim_{\varepsilon\to 0+} \dfrac{\varepsilon r(x,\lambda, \varepsilon)}{\lambda-m(x)+i\varepsilon m'(x)}=0,
\]
which gives $\lim\limits_{\varepsilon\to 0+} I_2(\lambda,\varepsilon)=0$. Similarly, 
$\lim\limits_{\varepsilon\to 0+} I_3(\lambda,\varepsilon)=0$

To calculate the limit of $I_1$, without loss of generality we assume that that $m'(x)>0$ for all $x\in [a,b]$;
otherwise one can change simultaneously the signs of $T$, $m$ and $n$.
Introduce a new variable $y=m(x)$. By the implicit function theorem
we have $x=\varphi(y)$ and $\varphi'(y)=1/m'(x)$, and
the expression for $I_1$ can be rewritten in the form
\[
I_1(\lambda,\varepsilon)=
\dfrac{1}{\pi}\,\int_{m(a)}^{m(b)} \dfrac{1}{\big(\overline z_1-\varphi(y)\big)\big( z_2-\varphi(y)\big)}
\dfrac{\varepsilon n\big(\varphi(y)\big)}{\big(\lambda-y\big)^2 + \dfrac{\varepsilon^2}{\varphi'(y)^2}}\,dy.
\]
Introducing another new variable $t$ by $y=\varepsilon t+\lambda$ we arrive at
\begin{equation*}
I_1(\lambda,\varepsilon)=
\dfrac{1}{\pi}\,\int_{\frac{m(a)-\lambda}{\varepsilon}}^{\frac{m(b)-\lambda}{\varepsilon}}
\cdot
\dfrac{1}{\big(\overline z_1-\varphi(\varepsilon t+\lambda)\big)\big( z_2-\varphi(\varepsilon t+\lambda)\big)}
\dfrac{n\big(\varphi(\varepsilon t+\lambda)\big)}{t^2 + \dfrac{1\mathstrut}{\varphi'(\varepsilon t+\lambda)^2}}\,dt.
\end{equation*}
One has the bounds
\[
\sup_{\frac{m(a)-\lambda}{\varepsilon}\le t\le \frac{m(b)-\lambda}{\varepsilon}}
\Big|n\big(\varphi(\varepsilon t+\lambda)\big)\Big|=\sup_{a\le x\le b}\big|n(x)\big|= N<+\infty
\]
and 
\[
\inf_{\frac{m(a)-\lambda}{\varepsilon}\le t\le \frac{m(b)-\lambda}{\varepsilon}} \dfrac{1}{\varphi'(\varepsilon t+\lambda)^2}=
\inf_{a\le x\le b} m'(x)^2=\kappa^2>0.
\]
This leads to the estimate
\[
\left| \dfrac{1}{\big(\overline z_1-\varphi(\varepsilon t+\lambda)\big)\big( z_2-\varphi(\varepsilon t+\lambda)\big)}\dfrac{n\big(\varphi(\varepsilon t+\lambda)\big)}{t^2 + \dfrac{1\mathstrut}{\varphi'(\varepsilon t+\lambda)^2}}\right|
\le \dfrac{N}{|\Im z_1 \Im z_2|(t^2+\kappa^2)}.
\]
Using the fact that the function $t\mapsto (t^2+\kappa^2)^{-1}$ belongs to $L^1(\RR)$
and applying the dominated convergence one obtains the equality
\begin{multline*}
\lim_{\varepsilon\to 0+} I_1(\lambda,\varepsilon)
=\dfrac{1}{\pi}\,\int\limits_{\lim\limits_{\varepsilon\to 0+}\frac{m(a)-\lambda}{\varepsilon}}^{\lim\limits_{\varepsilon\to 0+}\frac{m(b)-\lambda}{\varepsilon}}
\lim_{\varepsilon\to 0+}\dfrac{1}{\big(\overline z_1-\varphi(\varepsilon t+\lambda)\big)\big( z_2-\varphi(\varepsilon t+\lambda)\big)}
\cdot
\dfrac{n\big(\varphi(\varepsilon t+\lambda)\big)}{t^2 + \dfrac{1\mathstrut}{\varphi'(\varepsilon t+\lambda)^2}}\,dt.
\end{multline*}
For any $a\ne 0$ there holds
\[
\int_{-\infty}^{0}\dfrac{dt}{a^2+t^2}=\int_0^{+\infty}\dfrac{dt}{a^2+t^2}=\dfrac{1}{2}\int_{-\infty}^{+\infty}\dfrac{dt}{a^2+t^2}=\dfrac{\pi}{2|a|},
\]
and for any $c\in [a,b]$ one has
\[
\lim_{\varepsilon\to 0+}\frac{m(c)-\lambda}{\varepsilon}=\begin{cases}
+\infty, & \text{if }\lambda<m(c)\\
0 & \text{if }\lambda=m(c)\\
-\infty, & \text{if }\lambda>m(c)
\end{cases}.
\]
Finally, note that for $\lambda \in \big[m(a),m(b)\big]$ there holds
\begin{multline*}
\lim_{\varepsilon\to 0+}\dfrac{1}{\big(\overline z_1-\varphi(\varepsilon t+\lambda)\big)\big( z_2-\varphi(\varepsilon t+\lambda)\big)}\dfrac{n\big(\varphi(\varepsilon t+\lambda)\big)}{z^2 + \dfrac{1\mathstrut}{\varphi'(\varepsilon t+\lambda)^2}}\\
=\dfrac{1}{\big(\overline z_1-\varphi(\lambda)\big)\big( z_2-\varphi(\lambda)\big)}\,
\dfrac{n\big(\varphi(\lambda)\big)}{t^2+\dfrac{1\mathstrut}{\varphi'(\lambda)^2}}\,.
\end{multline*}
Putting all together we arrive at the equalities
\[
\lim_{\varepsilon\to 0+}k(\lambda,\varepsilon)=\begin{cases}
0, & \text{if }\lambda\notin m\big([a,b]\big),\\
\dfrac{\varphi'(\lambda)n\big(\varphi(\lambda)\big)}{2\big(\overline z_1-\varphi(\lambda)\big)\big( z_2-\varphi(\lambda)\big)},& \text{if } \lambda\in \big\{m(a),m(b)\big\},\\
\dfrac{\varphi'(\lambda)n\big(\varphi(\lambda)\big)}{\big(\overline z_1-\varphi(\lambda)\big)\big( z_2-\varphi(\lambda)\big)},& \text{if }\lambda\in m\big((a,b)\big).
\end{cases}
\]
Substituting these equalities into \eqref{eq-proj} and noting that, by assumption \eqref{eq-cond},
we have $E_T(\{m(a)\})=E_T(\{m(b)\})=0$ and $\mu(\{m(a)\})=\mu(\{m(b)\})=0$,
we arrive at
\begin{align*}
\langle f,Pg\rangle_\cH&=\int_{m([a,b])\cap S_T} 
\dfrac{\varphi'(\lambda)n\big(\varphi(\lambda)\big)}{\big(\overline z_1-\varphi(\lambda)\big)\big( z_2-\varphi(\lambda)\big)}\, d\mu (\lambda)\\
&=\int_{[a,b]\cap m^{-1}(S_T)}
\dfrac{1}{(\overline z_1-\lambda)( z_2-\lambda)} \, \dfrac{n(\lambda)}{m'(\lambda)}\, d\mu(m(\lambda)\big).
\end{align*}
Finally, using the inclusion $\supp \mu\subset S_T$, we can rewrite it as
\begin{equation}
        \label{eq-fpg1}
\langle f,Pg\rangle_\cH=
\Big\langle M(z_1)\xi_1, \int_a^b
\dfrac{1}{(\overline z_1-\lambda)( z_2-\lambda)} \, \dfrac{n(\lambda)}{m'(\lambda)}\, dE_T\big(m(\lambda)\big)
M(z_2){\xi_2}\Big\rangle_\cG.
\end{equation}

Now let us switch to the operator $\Pi$. Take a partition $\Delta$ of $[a,b]$, $a=\lambda_0<\lambda_1<\dots<\lambda_n=b$,
denote $\Delta_j:=[\lambda_{j-1},\lambda_j)$ and $\delta:=\max |\Delta_j|$, and consider the integral sums $\Phi_\Delta$
defined by \eqref{eq-phisum} with $\xi_j:=\lambda_j$. As $\Phi_\Delta$ converge in the norm sense to $\Phi$ as $\delta$ tends to $0$,
the products $\Phi_\Delta \Phi_\Delta^*$ converge to $\Pi$. On the other hand, using the orthogonality property:
$E\big(m(\Delta_j)\big)E\big(m(\Delta_k)\big)=0$ for $j\ne k$, we have the representation
\[
\Phi_\Delta \Phi_\Delta^*=\sum_{j=1}^n \dfrac{n(\lambda_j)}{m'(\lambda_j)}
\gamma(\lambda_j)E_T\big(m(\Delta_j)\big)\gamma(\lambda_j)^*.
\]
Therefore, with the help of the identity \eqref{eq-mgam} we arrive at
\begin{multline}
\begin{aligned}
\langle f, \Pi g\rangle_\cH&=\lim_{\delta\to 0}\sum_{j=1}^n
\dfrac{n(\lambda_j)}{m'(\lambda_j)}\Big\langle f, \gamma(\lambda_j)E_T\big(m(\Delta_j)\big)\gamma(\lambda_j)^*g\Big\rangle_\cG\\
&=
\lim_{\delta\to 0}\sum_{j=1}^n
\dfrac{n(\lambda_j)}{m'(\lambda_j)}\Big\langle \gamma(\lambda_j)^*\gamma(z_1)\xi_1, E_T\big(m(\Delta_j)\big)\gamma(\lambda_j)^*\gamma(z_2) {\xi_2}\Big\rangle_\cG\\
&=\lim_{\delta\to 0}\sum_{j=1}^n
\dfrac{n(\lambda_j)}{m'(\lambda_j)(\overline z_1 -\lambda_j)(z_2-\lambda_j)}
\end{aligned}\\
\cdot
\Big\langle
\big( M(z_1)-M(\lambda_j)\big)\xi_1,
E_T\big(m(\Delta_j)\big)
\big( M(z_2)-M(\lambda_j)\big){\xi_2}
\Big\rangle_\cG.
          \label{eq-lims}
\end{multline}
Denote
\begin{equation}
  \label{eq-mn}
  \begin{aligned}
n_0&:=\min_{x\in [a,b]}\big|n(x)\big|,&n_1&:=\max_{x\in [a,b]}\big|n(x)\big|, \\
m_0&:=\min_{x\in [a,b]}\big|m'(x)\big|, &m_1&:=\max_{x\in [a,b]}\big|m'(x)\big|,
& A&:=\dfrac{n_1 m_1}{m_0 n_0}.
\end{aligned}
\end{equation}
One has obviously $n_0>0$, $m_0>0$, and, for any $h\in\cG$,
\[
\big\|
E_T\big(m(\Delta_j)\big)M(\lambda_j) h
\big\|_\cG=
\Big\|
\dfrac{m(\lambda_j)-T}{n(\lambda_j)}\,E_T\big(m(\Delta_j)\big)h\Big\|_\cG
\le \dfrac{m_1 |\Delta_j|\cdot \big\|E_T\big(m(\Delta_j)\big)h\big\|_\cG}{n_0}.
\]
Using this estimate we obtain
\begin{multline*}
\bigg|
\lim_{\delta\to 0}\sum_{j=1}^n
\dfrac{n(\lambda_j)}{m'(\lambda_j)(\overline z_1 -\lambda_j)(z_2-\lambda_j)}
\cdot
\Big\langle
M(z_1)\,\xi_1,
E_T\big(m(\Delta_j)\big)
M(\lambda_j)\,\xi_2
\Big\rangle_\cG\bigg|\\
\begin{aligned}
\le& A
\lim_{\delta\to 0}\sum_{j=1}^n |\Delta_j |\cdot 
\Big\|E_T\big(m(\Delta_j)\big)M(z_1)\,\xi_1\Big\|_\cG
\cdot
\big\|E_T\big(m(\Delta_j)\big)\xi_2\big\|_\cG\\
\le&
A
\lim_{\delta\to 0}
\bigg(
\delta \sqrt{\sum_{j=1}^n \Big\|E_T\big(m(\Delta_j)\big) M(z_1)\,\xi_1\Big\|^2_\cG}
\cdot
\sqrt{\sum_{j=1}^n \|E_T\big(m(\Delta_j)\big)\xi_2\big\|^2_\cG}
\bigg)\\
=&
A
\lim_{\delta\to 0}
\bigg(
\delta \Big\|E_T\Big(m\big([a,b]\big)\Big)M(z_1)\,\xi_1\Big\|_\cG \cdot \Big\|E_T\Big(m\big([a,b]\big)\Big)\xi_2\Big\|_\cG\bigg)\\
=&0.
\end{aligned}
\end{multline*}
In a similar way,
\begin{multline*}
\lim_{\delta\to 0}\sum_{j=1}^n
\dfrac{n(\lambda_j)}{m'(\lambda_j)(\overline z_1 -\lambda_j)(z_2-\lambda_j)}
\cdot
\Big\langle
M(\lambda_j)\,\xi_1,
E_T\big(m(\Delta_j)\big)
M(z_2)\,\xi_2
\Big\rangle_\cG\\
=\lim_{\delta\to 0}\sum_{j=1}^n
\dfrac{n(\lambda_j)}{m'(\lambda_j)(\overline z_1 -\lambda_j)(z_2-\lambda_j)}
\cdot
\Big\langle
M(\lambda_j)\,\xi_1,
E_T\big(m(\Delta_j)\big)
M(\lambda_j)
\,\xi_2
\Big\rangle_\cG=0.
\end{multline*}
Injecting these estimates into \eqref{eq-lims} one arrives at
\begin{align*}
\langle f, \Pi g\rangle_\cH&=
\lim_{\delta\to 0}\sum_{j=1}^n
\dfrac{n(\lambda_j)}{m'(\lambda_j)(\overline z_1 -\lambda_j)(z_2-\lambda_j)}
\cdot \Big\langle M(z_1)\,\xi_1,
E_T\big(m(\Delta_j)\big)
M(z_2)\,{\xi_2}
\Big\rangle_\cG\\
&=\Big\langle 
\dfrac{m(z_1)-T}{n(z_1)}\,\xi_1,
\int_a^b \dfrac{n(\lambda)}{m'(\lambda) (\overline z_1-\lambda)(z_2-\lambda) } \,dE_T\big(m(\lambda)\big)
\dfrac{m(z_2)-T}{n(z_1)}\,{\xi_2}
\Big\rangle_\cG.
\end{align*}
Comparing this expression with \eqref{eq-fpg1} we conclude that
$P$ and $\Pi$ coincide in $\cD$. Therefore, the assertion of Lemma \ref{prop3} holds under the additional assumption \eqref{eq-mp0}.

Now let $[a,b]\subset J$ be an arbitrary interval satisfying \eqref{eq-cond}.
By Lemma \ref{lem1}, one can find an open  interval $(c,d)\subset[a,b]$ with the following properties:
\begin{itemize}
\item $c,d\notin\spec_p H$,
\item $m^{-1}(S_T)\subset J\subset (c,d)$,
\item $m'(x)\ne 0$ for all $x\in[c,d]$.
\end{itemize}
By the first part of the proof, we have
$E_H\big([c,d]\big)=\Pi\big([c,d]\big):=\Phi\big([c,d]\big)\Phi\big([c,d]\big)^*$.
On the other hand, the interval $[c,d]$ contains all  $\lambda\in [a,b]$ with 
$m(\lambda)\in\spec T$, or, equivalently,
all  $\lambda\in [a,b]\cap\spec H$.
This means that $E_T\big([a,c]\big)=E_T\big([d,b]\big)=\Pi\big([a,c]\big)=\Pi\big([d,b]\big)=0$, and,
finally,
\begin{multline*}
E_H\big([a,b])=
E_H\big([a,c])+E_H\big([c,d])+E_H\big([d,b])=E_H\big([c,d])\\
=\Pi\big([c,d])=\Pi\big([a,c])+\Pi\big([c,d])+\Pi\big([d,b])=
\Pi\big([a,b]).
\end{multline*}
This concludes the proof of Lemma \ref{prop3}.
\end{proof}

\section{Unitary equivalence}\label{sec-unt}

To avoid potential confusions let us introduce a definition.

\begin{defin}\label{def1}
Let $\cH_1$ and $\cH_2$ be Hilbert spaces. Let $\cL_j\subset \cH_j$
be closed linear subspaces viewed as Hilbert spaces with the induced scalar products,
and let $P_j:\cH_j\to\cL_j$ be the orthogonal projections, $j=1,2$.
We say that an operator $U\in \cL (\cH_1,\cH_2)$ 
\emph{defines a unitary map from $\cL_1$ to $\cL_2$}
if
\begin{itemize}
\item $U(\cL_1)\subset \cL_2$ and
\item the map $P_2 UP_1^*:\cL_1\to\cL_2$ is a unitary operator.
\end{itemize}
\end{defin}

\begin{lemma}\label{lem4}
Let $\cH_j$, $\cL_j$, $P_j$, $j=1,2$, be as in Definition \ref{def1}, 
and let an operator $U\in \cL(\cH_1,\cH_2)$ satisfy
\begin{equation}
 \label{eqpp12}
\text{\rm (i) }
U^* U= P_1^* P_1 \text{ and \rm (ii) }
U U^*=P_2^* P_2,
\end{equation}
then
\begin{itemize}
\item[\text{\rm(a)}]$U$ defines a unitary map from $\cL_1$ to $\cL_2$ and
\item[\text{\rm(b)}]$U P_1^*P_1= P_2^* P_2 U$.
\end{itemize}
\end{lemma}

\begin{proof} This is an elementary result and we give the proof just for the sake of completeness.
The part (b) follows directly from \eqref{eqpp12}, so we just need to prove the part (a).
Let us show first the inclusion $U(\cL_1)\subset \cL_2$. Let $x\in\cL_1$
such that $Ux\subset \cL_2^\perp$. In view of the definition of $P_2$ the last equality means that $P_2 Ux=0$.
On the other hand, with the help of \eqref{eqpp12} we obtain:
\begin{align*}
0=\|P_2 Ux\|^2_{\cL_2}&=\langle P_2 Ux,P_2 Ux\rangle_{\cL_2}\\
&=\langle x,U^* (P_2^* P_2) U x\rangle_{\cH_2}=\langle x, U (U^* U) U x \rangle_{\cH_2}=\|x\|^2_{\cH_1},
\end{align*}
which shows the equality $U(\cL_1)\cap \cL_2^\perp=\{0\}$ and the sought inclusion.

Now we have, with the help of \eqref{eqpp12},
\begin{align*}
(P_2 U P_1^*)(P_2 U P_1^*)^*&=P_2 U (P_1^* P_1) U^* P_2^*\\
&=P_2 U (U^* U) U^* P_2^*=P_2 P_2^* \equiv\text{Id}_{\cL_2},
\end{align*}
and 
\begin{align*}
(P_2 U P_1^*)^*(P_2 U P_1^*)&=P_1 U^* (P_2^* P_2) U P_1^* \\
&=P_1 U^* (U U^*) U P_1^*=P_1 P_1^* \equiv\text{Id}_{\cL_1},
\end{align*}
which shows the unitarity of $P_2 U P_1^*$.
\end{proof}

\begin{lemma}\label{lem6}
If a segment $[a,b]\subset J$ satisfies \eqref{eq-cond}, then:
\begin{itemize}
\item[\text{\rm(a)}] $\Phi([a,b]\big)^*\Phi([a,b]\big)=E_T\Big(m\big([a,b]\big)\Big)$
\item[\text{\rm(b)}] $\Phi([a,b]\big)$ defines a unitary map from $\ran E_T\Big(m\big([a,b]\big)\Big)\subset \cG$
and $\ran E_H\big([a,b]\big)\subset\cH$,
\item[\text{\rm(c)}] $E_H([a,b]\big)=\Phi([a,b]\big)E_T\Big(m\big([a,b]\big)\Big)\Phi([a,b]\big)^*$.
\end{itemize}
\end{lemma}

\begin{proof}
Taking into account Lemmas \ref{prop3} and \ref{lem4} we see that the assertions (b) and (c) follow if we prove (a).

Take again a partition $\Delta$ of $[a,b]$, $a=\lambda_0<\lambda_1<\dots<\lambda_n=b$,
denote $\Delta_j:=[\lambda_{j-1},\lambda_j)$ and consider the Riemann-Stieltjes integral sum \eqref{eq-phisum}
with $\xi_j:=\lambda_j$,
\[
\Phi_\Delta=\sum_{j=1}^n \sqrt{\dfrac{n(\lambda_j)}{m'(\lambda_j)}}\, \gamma(\lambda_j) E_T\big(m(\Delta_j)\big).
\]
As noted above, the sums $\Phi_\Delta$ converge in the operator norm to $\Phi:=\Phi\big([a,b]\big)$, the adjoint
operators $\Phi_\Delta^*$ converge then to $\Phi^*$, and the products
$\Phi_\Delta^*\Phi_\Delta$ converge to $\Phi^*\Phi$ as $\delta:=\max_j|\Delta_j|$ tends to $0$.
We have
\[
\Phi^*_\Delta \Phi_\Delta=\sum_{j,k=1}^n \sqrt{\dfrac{n(\lambda_j)}{m'(\lambda_j)}}
\sqrt{\dfrac{n(\lambda_k)}{m'(\lambda_k)}}\, E_T\big(m(\Delta_j)\big)\gamma(\lambda_j)^*\gamma(\lambda_k)\,E_T\big(m(\Delta_k)\big).
\]
Using \eqref{eq-mgam} we can write
\begin{equation*}
\gamma(\lambda_j)^*\gamma(\lambda_k)=L(\lambda_j,\lambda_k):=\begin{cases}
\dfrac{M(\lambda_k)-M(\lambda_j)}{\lambda_k-\lambda_j}, & j\ne k,\\
M'(\lambda_j)\equiv \dfrac{m'(\lambda_j)}{n(\lambda_j)} - \dfrac{n'(\lambda_j)}{n(\lambda_j)^2}\big(m(\lambda_j)-T\big),
& j=k.
\end{cases}
\end{equation*}
Note that the operators $L(\lambda_j,\lambda_k)$ commute with $T$ for any $j,k=1,\dots n$
and that we have the equality $E_T\big(m(\Delta_j)\big) E_T\big(m(\Delta_k)\big)=0$ for all $j\ne k$.
This simplifies the above expression:
\begin{align*}
\Phi^*_\Delta \Phi_\Delta&=\sum_{j=1}^n 
\dfrac{n(\lambda_j)}{m'(\lambda_j)} \, M'(\lambda_j) E\big(m(\Delta_j)\big)\\
&=
\sum_{j=1}^n 
\dfrac{n(\lambda_j)}{m'(\lambda_j)}\, \dfrac{m'(\lambda_j)}{n(\lambda_j)}\,  E\big(m(\Delta_j)\big)-
\sum_{j=1}^n 
\dfrac{n(\lambda_j)}{m'(\lambda_j)}
\dfrac{n'(\lambda_j)}{n(\lambda_j)^2}\big(m(\lambda_j)-T\big)E\big(m(\Delta_j)\big)\\
&=\sum_{j=1}^n E\big(m(\Delta_j)\big) -
\sum_{j=1}^n 
\dfrac{n(\lambda_j)}{m'(\lambda_j)}
\dfrac{n'(\lambda_j)}{n(\lambda_j)^2}\big(m(\lambda_j)-T\big)E\big(m(\Delta_j)\big)\\
&=E\Big(m\big([a,b]\big)\Big)-
\sum_{j=1}^n 
\dfrac{n(\lambda_j)}{m'(\lambda_j)}
\dfrac{n'(\lambda_j)}{n(\lambda_j)^2}\big(m(\lambda_j)-T\big)E\big(m(\Delta_j)\big).
\end{align*}
Using the constants \eqref{eq-mn} and $n_2:=\max_{x\in[a,b]}\big|n'(x)\big|$
one can estimate, for any $h\in \cG$,
\[
\Big\|
\big(m(\lambda_j)-T\big)E\big(m(\Delta_j)\big)h\Big\|_\cG\le m_1 |\Delta_j|\cdot \Big\|E\big(m(\Delta_j)\big)h\Big\|_\cG
\]
and, using the Cauchy-Schwartz inequality, we obtain
\begin{multline*}
\Big\|
\sum_{j=1}^n 
\dfrac{n(\lambda_j)}{m'(\lambda_j)}
\dfrac{n'(\lambda_j)}{n(\lambda_j)^2}\big(m(\lambda_j)-T\big)E\big(m(\Delta_j)\big)h\Big\|_\cG
\le \dfrac{n_2 m_1}{n_0 m_0}\sum_{j=1}^n |\Delta_j| \Big\|E\big(m(\Delta_j)\big)h\Big\|_\cG\\
\le \dfrac{n_2 m_1}{n_0 m_0} \sqrt{\sum_{j=1}^n |\Delta_j|}\cdot \sqrt{\sum_{j=1}^n |\Delta_j| \Big\|E\big(m(\Delta_j)\big)h\Big\|^2_\cG}\\
\le \dfrac{n_2 m_1 \sqrt{b-a}}{n_0 m_0} \max_j|\Delta_j| \sqrt{\sum_{j=1}^n \Big\|E\big(m(\Delta_j)\big)h\Big\|^2_\cG}\\
\le \dfrac{n_2 m_1 \sqrt{b-a}}{n_0 m_0} \max_j |\Delta_j|\, \Big\|E\Big(m\big([a,b]\big)\Big)h\Big\|_\cG.
\end{multline*}
Hence, $\Phi \Phi^*=\lim_{\max |\Delta_j|\to 0}
\Phi_\Delta \Phi^*_\Delta = E\Big(m\big([a,b]\big)\Big)$, which proves (a).
\end{proof}

\begin{lemma}\label{lem7}
Let intervals $[a,b]\subset J$ and $[c,d]\subset J$ be such that
\begin{itemize}
\item $a,b,c,d$ do not belong to $\spec_\mathrm{p} H$,
\item  $(a,b)\cap(c,d)=\emptyset$,
\end{itemize}
then  $\Phi\big([a,b]\big) E_T\Big( m\big([c,d]\big)\Big)=0$, 
$\Phi\big([a,b]\big)\Phi\big([c,d]\big)^*=0$ and
$\Phi\big([a,b]\big)^*\Phi\big([c,d]\big)=0$.
\end{lemma}

\begin{proof}
Take a partition $\Delta$ of $[a,b]$, $a=\lambda_0<\lambda_1<\dots<\lambda_n=b$,
denote $\Delta_j:=[\lambda_{j-1},\lambda_j)$ and consider the Riemann-Stieltjes integral sum \eqref{eq-phisum}
with $\xi_j:=\lambda_j$,
\[
\Phi_\Delta=\sum_{j=1}^n \sqrt{\dfrac{n(\lambda_j)}{m'(\lambda_j)}}\, \gamma(\lambda_j) E_T\big(m(\Delta_j)\big).
\]
Futhermore, take a partition $\Pi$ of $[c,d]$, $c=\mu_0<\mu_1<\dots<\mu_m=d$,
denote $\Pi_j:=[\mu_{j-1},\mu_j)$ and consider the Riemann-Stieltjes integral sum \eqref{eq-phisum}
with $\xi_j:=\mu_j$,
\[
\Phi_\Pi=\sum_{j=1}^m \sqrt{\dfrac{n(\mu_j)}{m'(\mu_j)}}\, \gamma(\mu_j) E_T\big(m(\Pi_j)\big).
\]
We know that $\Phi_\Delta$ and $\Phi_\Pi$ converge in the norm to $\Phi\big([a,b]\big)$
and $\Phi\big([c,d]\big)$ respectively as both $\delta:=\max_j|\Delta_j|$
and $\eps:=\max_j |\Pi_j|$ tend to $0$.

On the other hand, under the assumptions made we have $E_T\big(m(\Delta_j)\big)E_T\Big( m\big([c,d]\big)\Big)=0$
for any $j$, which implies $\Phi_\Delta E_T\Big( m\big([c,d]\big)\Big)=0$ and
$\Phi\big([a,b]\big)E_T\Big( m\big([c,d]\big)\Big)=0$.

In a similar way, using the equalities $E_T\big(m(\Delta_j)\big)E_T\big(m(\Pi_j)\big)=0$
for all $j=1,\dots,n$ and $k=1,\dots, m$, we obtain $\Phi_\Delta \Phi^*_\Pi=0$, which implies
$\Phi\big([a,b]\big)\Phi\big([c,d]\big)^*=0$.

Finally, we have, using \eqref{eq-mgam},
\begin{align*}
\Phi_\Delta^* \Phi_\Pi&=\sum_{j=1}^n \sum_{k=1}^m
\sqrt{\dfrac{n(\lambda_j)}{m'(\lambda_j)}}\cdot \sqrt{\dfrac{n(\mu_k)}{m'(\mu_k)}}
\, E_T\big(m(\Delta_j)\big) \gamma(\lambda_j)^*\gamma(\mu_k) E_T\big(m(\Pi_k)\big)\\
&= \sum_{j=1}^n \sum_{k=1}^m
\sqrt{\dfrac{n(\lambda_j)}{m'(\lambda_j)}}\cdot \sqrt{\dfrac{n(\mu_k)}{m'(\mu_k)}}
E_T\big(m(\Delta_j)\big) \dfrac{M(\mu_k)-M(\lambda_j)}{\mu_k-\lambda_j}
E_T\big(m(\Pi_k)\big)\\
&= \sum_{j=1}^n \sum_{k=1}^m
\sqrt{\dfrac{n(\lambda_j)}{m'(\lambda_j)}}\cdot \sqrt{\dfrac{n(\mu_k)}{m'(\mu_k)}}
\dfrac{M(\mu_k)-M(\lambda_j)}{\mu_k-\lambda_j}
E_T\big(m(\Delta_j)\big) 
E_T\big(m(\Pi_k)\big)=0,
\end{align*}
which proves the remaining equality $\Phi\big([a,b]\big)^*\Phi\big([c,d]\big)=0$.
\end{proof}

All the intervals we considered so far were contained in the gap $J$
together with their closures. Let us extend the above considerations to
arbitrary subintervals of $J$.

\begin{lemma}\label{lem9}
The strong limit 
\[
\Phi(J):=\slim_{\eps\to +0} \Phi\big([a_0+\varepsilon,b_0-\varepsilon]\big)
\]
exists and defines a unitary map
from $\ran E_T\big(m(J)\big)$ to $\ran E_H(J)$
\end{lemma}

\begin{proof}
Let us take a monotonically decreasing sequence of $(\varepsilon_n)_{n\ge 0}$ converging to 0
such that $a_0+\varepsilon_n,b_0-\varepsilon_n\notin\spec_\mathrm{p} H$. Denote $J_n:=[a_0+\varepsilon_n,b-\varepsilon_n]$
and
\[
\Phi(J_n\setminus J_{n-1}):=
\Phi\big([a_0+\varepsilon_n,a_0+\varepsilon_{n-1}]\big)
+
\Phi\big([b_0-\varepsilon_{n-1},b_0-\eps_n]\big).
\]
Let $\xi\in \cG$, then
\begin{equation}
     \label{eq-fser}
\Phi(J_n)\xi=\Phi(J_0)\xi +\sum_{j=0}^n \Phi(J_{n}\setminus J_{n-1})\xi.
\end{equation}
By Lemma \ref{lem7}, the summands on the right-hand side of \eqref{eq-fser} are paarwise orthogonal.
Moreover, by Lemma \ref{lem6}(a) we have
\[
\Big\|\Phi(J_0)\xi\Big\|^2_\cH=
\Big\|E_T \big(m(J_0)\big)\xi\Big\|^2_\cG
\text{ and }
\Big\|\Phi(J_{n}\setminus J_{n-1})\xi\Big\|^2_\cH=
\Big\|E_T \big(m(J_{n}\setminus J_{n-1})\big)\xi\Big\|^2_\cG,
\]
and once can estimate
\begin{align*}
\Big\|\Phi(J_0)\xi\Big\|^2_\cH
+\sum_{j=0}^n \Big\|\Phi(J_{n}\setminus J_{n-1})\xi\Big\|^2_\cH
&=\Big\|E_T \big(m(J_0)\big)\xi\Big\|^2_\cG+
\sum_{j=0}^n \Big\|E_T \big(m(J_{n}\setminus J_{n-1})\big)\xi\Big\|^2_\cG\\
&=\Big\|E_T\big(m(J_n)\big)\xi\Big\|^2_\cG\le \Big\|E_T\big(m(J)\big)\xi\Big\|^2_\cG<\infty.
\end{align*}
Hence, the series on the right-hand side of \eqref{eq-fser} converges, and
one shows that the limit is in fact independent of the choice of the sequence $\varepsilon_n$
in the standard way using the relations
\[
\slim_{\varepsilon\to +0} E_T\big((a_0,a_0+\varepsilon)\big)
=
\slim_{\varepsilon\to +0} E_T\big((b_0-\varepsilon,b_0)\big)
=0.
\]
Therefore, the map $\Phi(J)$ is well defined. Let us show that
the assumptions of Lemma \ref{lem4} are satisfied.
First, proceeding as above we can show the equality
\[
\Phi(J)^*:=\slim_{\varepsilon\to +0} \Phi\big([a_0+\varepsilon,b_0-\varepsilon]\big)^*.
\]
Second, by combining the results of Lemma \ref{prop3}, Lemma \ref{lem6}(a) and Lemma \ref{lem7}
we see that for $n<m$ one has the identities
$\Phi(J_n)^* \Phi(J_m)=E_T\big(m(J_n)\big)$  and 
$\Phi(J_n) \Phi(J_m)^*=E_H(J_n)$.
Taking first the strong limit for $m\to +\infty$ and then the strong limit as $n\to +\infty$, we arrive
at $\Phi(J)^* \Phi(J)=E_T\big(m(J)\big)$ and $\Phi(J) \Phi(J)^*=E_H(J)$,
which gives the conclusion by applying Lemma \ref{lem4}.
\end{proof}

The following theorem summarizes the above constructions.

\begin{theorem}\label{thm10}
For any borelian subset $\Omega$ of $J$ one has $E_H(\Omega)=\Phi(J) E_T\big(m(\Omega)\big)\Phi(J)^*$.
\end{theorem}

\begin{proof}
Assume first that $\Omega\subset J$  is a closed interval whose endpoints are not eigenvalues of $H$.
Take the same family of intervals $(J_n)$ as in the above proof of Lemma~\ref{lem9}, then
for sufficiently large $m$ and $n$ one has $\Omega\subset J_m$ and $\Omega\subset J_n$. 
By Lemma~\ref{lem6} and Lemma~\ref{lem7} we have
$E_H(\Omega)=\Phi(J_m) E_T\big(m(\Omega)\big)\Phi(J_n)^*$.
Taking the repeated strong limit, first for $n\to\infty$ and then for $m\to\infty$, one arrives
at $E_H(\Omega)=\Phi(J) E_T\big(m(\Omega)\big)\Phi(J)^*$. This identity is then extended to all borelian
subsets $\Omega\subset J$
using the $\sigma$-additivity.
\end{proof}

As an immediate corollary we have the following relation between $H$ and $T$:

\begin{corol}\label{corol11}
Introduce the orthogonal projections $P_H:\cH\to \ran E_H(J)$ and
$P_T:\cG\to \ran E_T\big(m(J)\big)$, then the operators $m(H_J)$
and $T_{m(J)}$ are related by
$m(H_J)=U T_{m(J)}U^*$,
where $U$ is the unitary operator from $\ran E_T\big(m(J)\big)$
to $\ran E_H(J)$ defined by $U:=P_H \Phi(J) P_T^*$.
\end{corol}

\begin{proof}
The unitarity of the map $U$ was already shown in Lemma \ref{lem9}, and
the requested equality is obtained from the spectral theorem and Theorem \ref{thm10} as follows
\begin{multline*}
m(H_J)=P_H \int_J m(\lambda)\, d E_H(\lambda) P_H^*\\
=P_H \Phi(J) \int_J m(\lambda) \, dE_T\big(m(\lambda)\big) \Phi(J)^*P_H^*=P_H \Phi(J) \int_{m(J)} \lambda \,dE_T(\lambda) \Phi(J)^*P_H^*\\
=P_H \Phi(J) P_T^* \int_{\RR} \lambda \,dE_{T_{m(J)}}(\lambda) P_T \Phi(J)^*P_H^*=U T_{m(J)} U^*. \qedhere
\end{multline*}
\end{proof}

\begin{remark}
The assertions of Corollary \ref{corol11} can also be rewritten in an equivalent form as
$H_J=U m^{-1}\big(T_{m(J)}\big)U^*$.
Taking into account the identity
\[
m^{-1}\big(T_{m(J)}\big)=\int_J \lambda dE\big(m(\lambda)\big)
\]
and applying the elementary arguments with the Riemann-Stieltjes integral sums
one can represent $H_J$ as a two-side operator Riemann-Stieltjes integral
\[
H_J=P_H\int_J \dfrac{\lambda n(\lambda)}{m'(\lambda)}\, \gamma(\lambda)dE_T\big(m(\lambda)\big)\gamma(\lambda)^*P_H^*,
\]
which may be viewed as a direct generalization of Proposition \ref{prop1} to the case of arbitrary spectra.
\end{remark}

\section{Applications}\label{sec-appl}

\subsection{Direct sums and arrays of quantum dots}

Let $L$ be a closed symmetric densely defined operator in a Hilbert space $\cK$ with deficiency indices $(1,1)$, and let
 $(\CC, \pi,\pi')$ be a boundary triplet for $L$; we denote by $\nu$ and $m$ be the associated
 $\gamma$-field and Weyl function, respectively. Denote by $K^0$ the self-adjoint extension
 of $L$ defined by the boundary conditions $\pi f=0$. 

Now let $\cA$ be a non-empty countable set. Introduce a Hilbert space $\cH$
and an operator $S$ in $\cH$ by
\[
\cH:=\bigoplus_{\alpha\in \cA}\cH_\alpha, \quad S=\bigoplus_{\alpha\in\cA}S_\alpha,
\]
where each $\cH_\alpha$ is a copy of $\cK$ and $S_\alpha:=L$. Elementary considerations show that $S$ is a closed densely defined
symmetric operator whose deficiency indices are $\big(a,a\big)$, where $a$ is the cardinality of the set $\cA$, and
as a boundary triplet for $S$ one can take
$(\cG,\Gamma,\Gamma')$ with
\[
\cG:=\ell^2(\cA), \quad
\Gamma (f_\alpha)=(\pi f_\alpha), \quad
\Gamma'(f_\alpha)=(\pi'f_\alpha).
\]
Clearly, the associated $\gamma$-field $\gamma(z)$
is just the direct sum, $\gamma(z)(\xi_\alpha)=\big(\nu(z)\xi_\alpha\big)$,
and the Weyl function is $M(z)=m(z)\Id$, where $\Id$ the identity operator in $\cG$. We refer to \cite[Section 3]{KM}
for a detailed discussion of boundary triplet machinery for direct sums. We also remark
that the above construction exhausts, up to a unitary equivalence, all the cases in which
the Weyl function is of scalar type, i.e. is just the multiplication by a scalar function, see \cite{ABMN}.

Our aim now is to study the  self-adjoint extensions of $S$.
Note first that, due to the above construction,
the distinguished extension of $S$ defined by $H^0:=S^*\uhr{\ker \Gamma}$ is just the direct
sum of the copies of the operator $K^0$ and, in particular, one has the equality
$\spec H^0=\spec K^0$.
Now let $T$ be a bounded self-adjoint operator in $\cG$. Consider the self-adjoint operator
\[
H_T:=S^*\uhr{\ker(\Gamma'-T\Gamma)};
\]
it is known that $H_T$ is a self-adjoint extension of $S$ \cite{DM}, and we are going to show
that its spectral analysis is covered by the scheme of the present paper.
Define
\[
\Gamma_T:=\Gamma, \quad \Gamma'_T:=\Gamma'-T\Gamma,
\]
then one can easily see that $(\cG,\Gamma_T,\Gamma'_T)$ is another boundary triplet for $S$, that
the associated $\gamma$-field is the same as for $(\cG,\Gamma,\Gamma')$,
and that the associated Weyl function $M_T$ has the form
$M_T(z):=m(z)-T$. At the same time, with respect to this new boundary triplet
the operator $H^0$ corresponds to the boundary conditions $\Gamma_T f=0$,
and $H_T$ corresponds to the boundary conditions $\Gamma'_Tf=0$.

The physical interpretation of the above situation is as follows. The operator
$K^0$ may be viewed as a Hamiltonian of a single quantum dot, and 
a copy of the quantum dot is placed at each node of a discrete structure $\cA$.
Then the operator $H^0$  describes the array of non-interacting quantum dots, while
$H_T$ is Hamiltonian of the quantum dots interacting with each other
through the boundary conditions
$\Gamma'f=T\Gamma f$, and one can be interested in the dependence of the spectral properties of $H_T$
on the ``interaction parameter'' $T$ describing the inter-node interactions. Such an approach is known under the name
``restriction-extension procedure'' which proved its usefulness in the study
of solvable models in quantum mechanics \cite{AGHH,DEG,pavlov}.
So Corollary \ref{corol11} being applied to the present situation gives the following results:
\begin{theorem}
For any bounded self-adjoint operator $T$ in $\ell^2(\cA)$ and any interval $J\subset \RR\setminus\spec K^0$
one has $(H_T)_J=U m^{-1}\big(T_{m(J}\big)U^*$,
where $U$ acts as
\[
U =\int_J \sqrt{\dfrac{1}{m'(\lambda)}}\, \gamma(\lambda) dE_T\big(m(\lambda)\big)
\] 
being considered as a unitary map from $\ran E_T\big(m(J)\big)$ to $\ran E_{H_T}(J)$.
\end{theorem}

For example, the papers \cite{GP1,GP2} study the case when $K^0$ is a magnetic harmonic oscillator,
\[
\cK=L^2(\RR^2), \quad K^0=-\Big(\dfrac{\partial}{\partial x_1} -iB x_2\Big)^2 -\dfrac{\partial^2}{\partial x_2^2} +\omega^2 (x_1^2+x_2^2), \quad
\omega,B>0,
\]
and $L$ is the restriction of $K^0$ on the functions vanishing at the origin. Under a certain standard choice
of the boundary triplet, the associated maps $\nu(z)$ and $m(z)$ are of the form
\[
\nu(z)\xi=\xi G(\cdot,0;z), \quad
m(z)=\lim_{x\to 0}\Big(
G(x,0;z)+\dfrac{\log|x|}{2\pi}
\Big)
\]
where $G$ is the Green function of $K^0$, i.e. the integral kernel of the resolvent $(K^0-z)^{-1}$
for $z\notin\spec K^0$; explicit analytic expressions for $G$ and $m$ are given in \cite{GP1}.
The operator $K^0$ has a discrete spectrum, which means that $\spec H^0$ is a discrete set.
Outside this ``forbidden set'' the Hamiltonian $H_T$ is completely described in terms of $T$.

\subsection{Differential operators on networks}

Another series of examples comes from the study of differential operators on networks, also called
metric graphs or quantum graphs, see the recent monograph \cite{BK} and the collections \cite{AGA,P12}. Let us describe a representative particular
case where our machinery works; we use the notation proposed in \cite{CW}.

Let $X$ be a countable connected graph with symmetric neighborhood relation $\sim$
and without loops and multiple edges. We shall view it as a one-complex, where each egde
is a homeomorphic copy of the unit interval,  and the edges are glued together
at common endpoints. We write $X^0$ for the vertex set and $X^1$ for the one-skeleton of $X$.
Every point of $X^1$ is of the form $(xy,t)$, the point at distance $t$ from $x$ on the non-oriented edge
$[x,y]=[y,x]$, where $t\in[0,1]$ and $x,y,\in X^0$ with $x\sim y$. Thus
one has $(xy,0)=x$ and $(xy,t)=(yx,1-t)$. In this way, the discrete graph metric
on the vertex set (minimal length = number of edges of a connecting path)
has a natural extension to $X^1$. 

On $X^0$ one has a natural discrete measure $m^0$ defined by
\[
m^0(x)=\#\{y\in X^0:\,y\sim x\}
\]
and we assume that 
\[
m^0(x)<\infty \text{ for any } x\in X^0.
\]
On $X^1$ we introduce the continuous Lebesgue measure $m^1$ which at
the point $(xy,t)$ is given by $dt$ if $t\in(0,1)$, and the vertex set
has $m^1$-measure equal to zero. The graph $X$ equipped with the above constructions
will be called a network or a metric graph.

For any function $F:X^1\to\CC$ and for $x\in X^0$ and $y\sim x$ denote by $F_{xy}$ the function
$t\mapsto F(xy,t)$, then the Hilbert space $L^2(X^1,m^1)$ is exactly the space of the measurable
functions $F$ such that
\[
\|F\|^2_{L^2(X^1,m^1)}:=\sum_{x\sim y}\|F_{xy}\|^2_{L^2(0,1)}<\infty;
\]
we assume that each edge appears only once in the sum.
Associated with a network, there are at least two natural operators.
The first one is the discrete transition operator $P$ acting on functions $g:X^0\to \CC$
by
\begin{equation*}
P g(x)=\dfrac{1}{m^0(x)}\sum_{y:y\sim x} g(y).
\end{equation*}
The above expression defines a bounded self-adjoint operator in $\ell^2(X^0,m^0)$
with the norm $\le 1$, to be denoted by the same symbol $P$.
The second operator is the continuous (positive) Laplace operator $L$ acting in the space
$L^2(X^1,m^1)$ as the second derivative. More precisely,
for any function $F$ the prime sign will denote the derivation
with respect to the length parameter, i.e.
\[
F'(xy,t):=F'_{xy}(t).
\]
We introduce the space
\[
\Hat H^k(X^1,m^1):=\Big\{
F\in \ELL^2(X^1,m^1):\, F^{(k)}\in \ELL^2(X^1,m^1)
\Big\},
\]
then the operator $L$ acts as
$L F=-F''$ on functions
$F\in \Hat H^2(X^1,m^1)$
 satisfying
the boundary conditions
\begin{gather} 
              \label{eq-fcont}
F(xu,0)=F(xv,0)=: F(x) \text{ for all } x,u,v\in X^0 \text{ with } u\sim x \text{ and } v\sim x,\\
F'(x)=0 \text{ for all } x\in X^0, \quad
\text{where }
F'(x):=\sum_{y:y\sim x} F'(xy,0+).
\nonumber
\end{gather}
It has been  known for a long time that the discrete operator $P$ and the network laplacian $L$ are closely related \cite{vB,E97,nic,C97}.
In particular, for the case of a finite $X$ it was shown in \cite{vB} that
\[
\spec L \setminus \Sigma=\big\{
z\notin\Sigma: \, \cos\sqrt z\in\spec P
\big\} \quad \text{with}\quad \Sigma:=\big\{(\pi n)^2:\, n\in\NN\big\},
\]
and it was extended in \cite{C97} to the case of arbitrary graphs.
The set $\Sigma$ plays a special role, and its relation with the spectrum of $L$
is also known but the presentation of the respective results would need
some special vocabulary from the graph theory; an interested reader may consult the respective section in \cite{C97}.
 It was pointed out in the author's paper \cite{KP06}
that the boundary triplet machinery can be applied to the study of the operator $L$.
This approach was used to improve the above relation: it was first shown that
\[
\spec_\star L\setminus\Sigma=\big\{ 
z\notin\Sigma: \, \cos\sqrt z\in\spec_\star P\big\},
\quad
\star\in\{\text{p},\text{pp},\text{disc},\text{ess},\text{ac},\text{sc}\},
\]
see \cite[Section 3.5]{BGP08}, and later it was shown that for any interval
$J\subset\RR\setminus\Sigma$ the operator $\cos\sqrt{L_J}$ is unitarily equivalent to
$P_{\cos\sqrt J}$, see \cite[Theorem 17]{KP11}; here and below we use the notation
\[
\cos\sqrt J:=\{\cos\sqrt \lambda:\, \lambda\in J\}.
\]
Let us recall
how the boundary triplet machinery applies to this case and explain what kind of improvements can be obtained.

Denote by $S$ the restriction of $L$ to the functions $F$ for which $F(x)=0$ for all $x\in X^0$.
Clearly, this is a closed densely defined symmetric operator. The following assertions are proved in \cite{KP06}:
\begin{itemize}
\item the adjoint operator $S^*$ acts as $F\mapsto -F''$ on the domain
\[
\dom S^*=\Big\{
F\in \Hat H^2(X^1,m^1): \text{Eq. \eqref{eq-fcont} holds}
\Big\},
\]
\item  as a boundary triplet for $S$ one can take $(\cG,\Gamma,\Gamma')$ with
\[
\cG=\ell^2(X^0,m^0), \quad
\Gamma f:=\big( F(x)\big), \quad
\Gamma' f:=\Big(\dfrac{F'(x)}{m^0(x)}\Big),
\]
\item the associated $\gamma$-field is given by
\[
\big[\gamma(z)\xi\big] (xy,t)=\xi(x) \dfrac{\sin \big(\sqrt z(1-t)\big)}{\sin \sqrt z}+
\xi(y) \dfrac{\sin (\sqrt z t)}{\sin \sqrt z},
\]
\item the associated Weyl function takes the form
\[
M(z)=-\dfrac{\sqrt z}{\sin\sqrt z} \Big( \cos\sqrt z -P).
\]
\end{itemize}
It is easy the see that:
\begin{itemize}
\item $L$ is exactly the restriction of $S^*$ to $\ker\Gamma'$,
\item the restriction $L^0:=S^*\uhr\ker \Gamma$ is the direct sum  (over all edges $xy$)
of the Dirichlet Laplacians on $(0,1)$, and $\spec L^0=\Sigma$.
\end{itemize}

So we are in the situation covered by the present paper. Using the equality $n=2m'$
and applying Corollary \ref{corol11} we arrive at
the following result:
\begin{theorem}
For any interval $J\subset \RR\setminus\Sigma$ one has
$\cos\sqrt{L_J}=U P_{\cos\sqrt J}\, U^*$, where the operator
$U$ acts by
\[
U =\sqrt 2 \int_J \gamma(\lambda) dE_P(\cos\sqrt\lambda).
\]
and defines a unitary map from $\ran E_P(\cos\sqrt J)$ to $\ran L_J$.
\end{theorem}

We note that the above operator $L$ is just a particular case of a differential operator 
acting on a network. It was shown in several papers that Weyl fonctions of the form
\eqref{eq-mspec} appear also in other cases, for example, for operators with magnetic
fields, more general boundary conditions, additional scalar potentials, we refer
e.g. to \cite{BGP07,KP06,KP11,P08} and to the references there-in. In all those cases
one obtains a similar result on a unitary equivalence between a differential operator
and a certain discrete operator: the operator $P$ becomes then a generalized
discrete Laplacian \cite{P08}, and the function $z\mapsto\cos\sqrt z$ should be replaced
by a certain analytic function expressed in terms of fundamental solutions \cite[Section 3.2]{KP11}.

\begin{remark}
The above operator $P$ is just one operator from a large family of Laplace-type operators associated with a graph.
There exists various unbounded counterparts for which the problem of self-adjoint extensions
has its own interest, cf. the recent works \cite{gol,jorg1}.
\end{remark}

\begin{remark}
The operators $L$ of the above type could be viewed as operator-valued ordinary differential
operators, whose study is rather active during the last years
\cite{vBM,GG,GWZ,MN,mogil}, but are not aware of any previous result giving a unitary equivalence
between such an operator and the coefficients in the respective boundary conditions.
\end{remark}

\section*{Acknowledgments}

The work was partially supported by ANR NOSEVOL 2011 BS01019 01
and GDR Dynamique quantique.

\end{document}